\documentclass[preprint,12pt]{imsart}
\usepackage[margin=1in]{geometry}
\RequirePackage{natbib}
\RequirePackage[OT1]{fontenc}
\RequirePackage{amsthm,amsmath}
\RequirePackage{natbib}
\RequirePackage{hyperref}
\usepackage{bbm}
\usepackage{amsmath,amssymb, amsthm}
\usepackage{amsfonts}

\numberwithin{equation}{section}
\theoremstyle{plain}
\newtheorem{thm}{Theorem}[section]
\newtheorem{corollary}{Corollary}[section]

\newtheorem{example}{Example}[section]
\newtheorem{proposition}{Proposition}[section]
\newtheorem{lemma}{Lemma}[section]
\newtheorem{definition}{Definition}
\newtheorem*{definition*}{Definition}
\newtheorem{remark}{Remark}[section]
\newcommand{\X}{\mathcal{X}}
\newcommand{\A}{\mathcal{A}}

\newcommand{\B}{\mathcal{B}}

\newcommand{\G}{\mathcal{G}}

\newcommand{\EE}{\mathcal{E}}

\newcommand{\E}{\mathbb{E}}
\newcommand{\PP}{\mathbb{P}}
\newcommand{\II}{\mathbbm{1}}
\newcommand{\ol}{\overline}
\newcommand{\field}[1]{\mathbb{#1}}
\newcommand{\R}{\field{R}}
\newcommand{\C}{\field{R}}
\newcommand{\Var}{\text{Var}}

\begin{document}
\title{Concentration Inequalities in Locally Dependent Spaces}
\runtitle{Concentration Inequalities in Locally Dependent Spaces}

\begin{aug}
\author{\fnms{Daniel} \snm{Paulin} \ead[label=e1]{paulindani@gmail.com}} 
\runauthor{D. Paulin}

\affiliation{National University of Singapore}

\address{
Department of Mathematics, National University of Singapore\\
10 Lower Kent Ridge Road, Singapore 119076, Republic of Singapore.\\
\printead{e1}}

\end{aug}



\begin{abstract}
This paper studies concentration inequalities for functions of locally dependent random variables.
We show that the usual definition of local dependence does not imply concentration for general Hamming Lipschitz functions. We define hypergraph dependence, which is a special case of local dependence, and show that it implies concentration if the maximal neighborhood size is small. We prove concentration in Hamming distance, Talagrand distance, and for self-bounding functions of a particular type under this dependence structure.
\end{abstract}
\maketitle

\section{Introduction}
Local dependence, when the variables only depend on those others which are in their neighborhood, has been one of the first examples of Stein's method, see \citep{ChenShaoLoc} and the references therein.

The usual form of local dependence is the following (based on \citep{ChenGoldsteinShao}, Chapter 4.7.):
\begin{definition}\label{LD1def}
A group of random variables $\{X_i, i\in \A\}$ satisfies $(LD1)$ if for each $i\in \A$ there exists $A_i\in [n]$ such that $X_i$ and $X_{A_i^c}$ are independent. 

An undirected graph $\G=([n],\EE)$ is called a dependency graph of $\{X_i, i\in \A\}$ if each $X_i$ can only depend on its neighbors in $\G$ (i.e. it is independent of the complement of its neighbors). An example for such a graph $\G$ is a graph with edge between $i$ and $j$ if $i\in A_j$ or $j\in A_i$ (i.e. one of them is in the neighborhood of the other).

We say that $\{X_i, i\in [n]\}$ satisfies \emph{(LD1, $m$)} if there is a $\G$ dependency graph that has maximum degree at most $m-1$.
\end{definition}

Let $\G=(V,E)$ be an undirected graph. The \emph{chromatic number} of $\G$, $\chi(\G)$ is the smallest positive integer $k$ such that the vertices of $\G$ can be colored with $k$ colors with no edge between vertices of the same color.

\citep{Jansonlargedev} shows concentration of sums under (LD1), and shows that Chernoff-Hoeffding and Bernstein inequalities also hold for sums of (LD1) dependent variables, with constants less than $\chi(\G)$ times weaker than in the independent case.

The objective of this paper is to investigate whether this result holds for more general functions of (LD1) dependent variables. We could extend \citep{Jansonlargedev} to subadditive functions. On the other hand, as the following counterexample shows, concentration does not holds for arbitrary Hamming-Lipschitz functions:

\begin{example}\label{counterex}
For $n$ even, let $X_1,\ldots,X_{n/2}, X_1',..., X_{n/2}'$ be random variables taking values 1 and -1.
Let $X_1,...,X_{n/2}$ be i.i.d. with $P(X_i=1)=P(X_i=-1)=1/2$. 
Let $Q$ be an independent random variable with $P(Q=1)=P(Q=-1)=1/2$. Define $X_i'=Q\cdot X_i$, $1\le i\le n/2$.

For $\{X_1,\ldots,X_{n/2}, X_1',..., X_{n/2}'\}$ as defined such satisfy (LD1,2).
Define the function $g:\R^2\to \R$ as $g(a,b)=a\cdot b/2$, then $g(1,1)=g(-1,-1)=1/2$ and $g(1,-1)=g(-1,1)=-1/2$. 
Define $f(X_1,...,X_{n/2},X_1',...,X_{n/2}')=\sum_{i=1}^{n/2} g(X_i,X_i')$, then this $f$ is 1-Hamming Lipschitz (depending on $n$ variables).
On the other hand, for the distribution we gave to $X_i$ and $X_i'$, we have $g(X_i,X_i')=Q$, so  
$f(X_1,...,X_{n/2},X_1',...,X_{n/2}')=nQ/2$, taking values $n/2$ and $-n/2$ with probability 1/2.
\end{example}

We have looked for examples in the literature about (LD1) dependent random variables, and most of them were defined as functions of independent random variables. For such cases, as we will show, concentration inequalities hold for general functions.

\subsection{Main definitions}

We will use the fractional chromatic number:
\begin{definition*}
Let $\G=(V,E)$ be an undirected graph. The \emph{fractional chromatic number} of $\G$, $\chi^*(\G)$ is the smallest positive real $k$ for which there exists a probability distribution over the independent sets of $\G$ such that for each vertex $v$, given an independent set $S$ drawn from the distribution,
\[Pr(v\in S) \ge \frac{1}{k}.\]
\end{definition*}
The independent sets of $\G$ here mean all the subsets of the vertices of $\G$ that contain no edges between them.

\citep{Jansonlargedev} introduces these:
\begin{definition*}
Given $\A$ and $\{X_{\alpha}\}$, $\alpha\in \A$, we make the following definitions:
\begin{itemize}
\item A subset $\A'$ of $\A$ is independent if the corresponding random variables $\{X_{\alpha}\}_{\alpha\in A'}$ are independent.
\item A family $\{\A_j\}_j$ of subsets of $\A$ is a cover of $\A$ if $\cup_j \A_j=\A$.
\item A family $\{(A_j,w_j)\}_j$ of pairs $(\A_j,w_j)$ where $\A_j\subset \A$ and $w_j\in [0,1]$
is a \emph{fractional cover} of $\A$ if $\sum_{j:\alpha\in \A_j}w_j\ge 1$ for each $\alpha\in \A$.
\item A (fractional) cover is \emph{proper} if each set $A_j$ is independent.
\item $\chi(\A)$ is the size of the smallest proper cover of $\A$, i.e. the smallest $m$ such that $\A$ is the union of $m$ independent subsets
\item $\chi^*(\A)$ is the minimum of $\sum_{j}w_j$ over all proper fractional covers
$\{(\A_j,w_j)\}_j$.
\item We say that a fractional cover $\{(\A_j,w_j)\}_j$ is \emph{exact} if $\sum_j w_j \II_{\A_j}=\II_{A}$.
\end{itemize}
\end{definition*}

It is shown in \citep{Jansonlargedev} that for (LD1,$m$) with dependency graph $\G$,
\[\chi^*(\A)\le \chi^*(\G)\le \chi(\G)\le m.\]

Lemma 3.2 of \citep{Jansonlargedev} shows that we can make an exact fractional cover from any fractional cover without changing $\sum_{j}w_j$, thus we can restrict our attention to exact fractional covers.

The first result for sums is the following:

\begin{thm}[Theorem 2.1 of \citep{Jansonlargedev}]
Suppose that $\{X_i\}_{i \in \A}$ satisfies (LD1), $a_i\le X_i\le b_i$ for $i\in \A$ and some real numbers $a_i$ and $b_i$. Then, for $t>0$,
Denote $S:=\sum_{i\in \A} X_i$, then for $t>0$,
\begin{equation}\label{thm21eq}
\PP(X\ge \E X+t)\le \exp\left(-2\frac{t^2}{\chi^*(\A) \sum_{i\in \A}(b_i-a_i)^2}\right).
\end{equation}
The same estimate holds for $\PP(X\le \E X-t)$.
\end{thm}

Further Bernstein-type results are proven in \citep{Jansonlargedev}, which take into account  the variance of $X_i$, and thus give better bounds than \eqref{thm21eq} for sums of random variables with small variances, an example is the following:
\begin{thm}[Theorem 2.3 of \citep{Jansonlargedev}]
Suppose that $\{X_i\}_{i \in \A}$ satisfies (LD1), $X_i-\E X_i \le b$ for some $b>0$, and all $i\in \A$. Then, for $t>0$,
Denote $S:=\sum_{i\in \A} \Var (X_i)$, then for $t>0$,
\begin{equation}\label{thm21eq}
\PP(X\ge \E X+t)\le \exp\left(-\frac{8t^2}{25\chi^*(\A)(S+bt/3)}\right).
\end{equation}
\end{thm}

The main idea of the proofs in \citep{Jansonlargedev} is to separate $\sum_{i\in \A}X_i$ into sums of independent random variables, and use the concentration properties of such sums.

Our first result, Theorem \ref{thmsubadd}, gives an upper tail bound for subadditive functions. The proof is based on the same idea as \citep{Jansonlargedev}. An application is given: an estimate for the upper tail of the norm of random matrix with locally dependent entries.

Example \ref{counterex} made us look for other, stronger definitions of local dependence, that are sufficient for concentration for large class of functions ((GD) is based on \citep{ChenGoldsteinShao}).

\begin{definition}[GD]\label{GDdef}
$\{X_i, i\in \A\}$ satisfies \emph{graphical dependence} if we can define a graph $\G=(\A,\EE)$ such that
for any pair of disjoint sets $\Gamma_1, \Gamma_2$ in $\A$ such that there is no edge in $\EE$ that has one endpoint in $\Gamma_1$ and another in $\Gamma_2$, the sets of random variables $X_{\Gamma_1}$ and $X_{\Gamma_2}$ are independent. In this case $\G$ is called the dependency graph. We say $\{X_i, i\in [n]\}$ satisfies \emph{(GD, $m$)} if $\G$ has maximum degree at most $m-1$.
\end{definition}

\begin{definition}[HD]\label{HDdef}
Let $\{Y_i, i\in [N]\}$ be a set of independent random variables taking values in $\Sigma=\Sigma_1\times \ldots \times \Sigma_N$, and for each $i\in [n]$, let $S_i$ be subsets of $[N]$, and $X_i: Y_{S_i}\to \Lambda_i$ be random variables depending on $Y_{S_i}$. For each $j\in [N]$, let $R_j:=\{i\in [N] s.t. j\in S_i\}$ (i.e. $R_j$ the set of $X_i$ depending of $Y_j$, and $S_i$ is the set of $Y_j$ that $X_i$ depends on).

We say that $\{X_i, i\in [n]\}$ satisfies \emph{(HD, $k$, $l$)} if $|S_i|\le k$ and $|R_j|\le l$ for every $i\in [n], j\in [N]$.
Let $\G=([n],\EE)$ be an undirected graph with an edge between $i$ and $j$ if $X_i$ and $X_j$ depend on some common $Y_k$ (i.e. if
$S_i\cap S_j\neq \emptyset$). If $\G$ has maximum degree at most $m-1$, then we say $\{X_i:i\in [n]\}$ satisfies \emph{(HD, $m$)}.
\end{definition}
A relation between (HD, $k$, $l$) and (HD, $m$) is given by the following lemma (the proof is given in Section \ref{secproofs}):
\begin{lemma}\label{HDlemma}
Suppose that $\{X_i:i\in [n]\}$ satisfies $(HD,m)$ for some $\{Y_i, i\in [N]\}$, then we can define $\{Y_i', i\in [N']\}$ such that $\{X_i:i\in [n]\}$ satisfies $(HD,m,m)$ for $\{Y_i', i\in [N']\}$.
\end{lemma}
\begin{example}[$m$-dependence]
A simple example to illustrate the difference between (HD,$k$,$l$) and (HD,$m$) is the following: let $Y_1,\ldots,Y_n$ be independent random variables, and
\[X_1:=f_1(Y_1,\ldots,Y_m), X_2:=f_2(Y_2,\ldots, Y_{m+1}), \ldots, X_n:=f_n(Y_n,Y_1,\ldots,Y_{m-1}).\]
Then one can easily prove (by breaking $Y$ into groups of size $m$) that $X_1,\ldots,X_n$ satisfy (HD, $2$, $2m-1$) and (HD,$2m-1$).
\end{example}
\begin{example}[Triangles in Erd\H{o}s-R\'{e}nyi graph]
Let $G(n,p)$ be an Erd\H{o}s-R\'{e}nyi graph, with edges $(X_{ij})_{1\le i<j\le n}$ (i.e. $X_{ij}$ is the indicator function of the edge between $i$ and $j$), and denote by $(T_{ijk})_{1\le i<j<k\le n}$ the indicator functions of the triangles between vertices $i,j,k$. Then one can easily see that $(T_{ijk})_{1\le i<j<k\le n}$ satisfies (HD,$n-2$,$3$) and (HD,$3(n-3)$).
\end{example}

It is an easy exercise to prove that 
\[(\text{HD, } m) \Rightarrow (\text{GD, }m)\Rightarrow  (\text{LD1, }m).\]

The reverse implications are false in general. 

(LD1, $m$) does not imply (GD, $m$), as we can see from Example \ref{counterex}.

(GD, $m$) does not imply (HD, $m$), we can see this from the example in \citep{onedep} where they construct a one - dependent sequence (future independent of past) which only satisfies (GD, $m$) (the existence of such a sequence was an open question for many years).

It remains an open question whether (GD, $m$) implies concentration inequalities for general functions. At the moment, we do not know of practical applications that satisfy (GD, $m$), but not (HD, $m$).

\subsection{Self-bounding and $\alpha$-self-bounding functions}
Self bounding functions were introduced in \citep{Boucheronsharp}, and found many applications. In \citep{LugosiSelfBounding}, the authors introduce $(a,b)$ self-bounding and weakly $(a,b)$ self-bounding functions.

For independent random variables, $(a,b)$ self-bounding functions are a large class of functions, that contain Hamming Lipschitz functions, configuration functions, suprema of positive valued empirical processes. They also imply Talagrand's convex distance inequality.

In \citep{Martoncoupling}, we have defined a stronger condition, $\alpha$-$(a,b)$ self-bounding and weakly
$\alpha$-$(a,b)$ self-bounding functions, and shown that such functions satisfy concentration inequalities under some dependence condition. As we are going to see, they also satisfy concentration inequalities under the (HD,$k$,$l$) dependence condition.

The following definitions of self-bounding functions are from \citep{LugosiSelfBounding} (we made a slight generalization, they had $\Lambda_1=\ldots=\Lambda_n=\X$).

For $1\le i\le n$, and $x\in \Lambda$, let $x_{-i}=(x_1,\ldots,x_{i-1},x_{i+1},\ldots,x_n)$, and let
$\Lambda_{-i}:=\Lambda_1\times\ldots\times\Lambda_{i-1}\times \Lambda_{i+1}\times\ldots\times \Lambda_n$.
\begin{definition}\label{defsb}
A function $g:\Lambda\to \R$ is called $(a,b)$-self-bounding for some $a,b\ge 0$ if there are functions $g_i:\Lambda_{-i}\to \R$ such that for all $i=1,\ldots,n$ and all $x\in \Lambda$,
\begin{enumerate}
\item$0\le g(x)-g_i(x_{-i})\le 1$, and\\
\item$\sum_{i=1}^n (g(x)-g_i(x_{-i}))\le ag(x)+b$.
\end{enumerate}
A function $g:\Lambda\to \R$ is called weakly $(a,b)$-self-bounding if there are functions $g_i:\Lambda_{-i}\to \R$ such that for all $x\in \Lambda$,
\[\sum_{i=1}^n \left(g(x)-g_i(x_{-i})\right)^2\le a g(x)+b.\]
\end{definition}
\begin{remark}
If $g$ is $(a,b)$-self-bounding, then it is also $(a,b)$-weakly-self-bounding.
If $g$ is $(a,b)$-self-bounding for some $g_i$, then it is also $(a,b)$-self-bounding for
\begin{equation}\label{fimin}
g_i(x_{-i}):=\inf_{x_i'\in \Lambda_i}f(x_1,\ldots,x_{i-1},x_{i}',x_{i+1},\ldots,x_n).
\end{equation}
If $g$ is weakly $(a,b)$-self-bounding, then in this paper we will also assume that $g_i(x_{-i})\le g(x)$ for all $x\in \Lambda$, and in this case, we can choose $g_i$ as in \eqref{fimin}. In the rest of this paper, we assume that $g_i$ is chosen as \eqref{fimin}.
\end{remark}

For these functions, the following concentration inequalities hold (\citep{LugosiSelfBounding} supposed that $\Lambda_1=\ldots=\Lambda_n=\X$, but the same results trivially hold for this case):
\begin{thm}\label{Lugosi}(\citep{LugosiSelfBounding})
Let \[X:=(X_1,\ldots,X_n)\] be a vector of independent random variables, taking values in $\Lambda$ and let $f: \Lambda \to \R$ be a non-negative measurable function such that $Z=f(X)$ has finite mean. For $a,b\ge 0$, define $c=(3a-1)/6$.
If $f$ is $(a,b)$-self-bounding, then for all $\lambda\ge 0$,
\[\log \E\left[e^{\lambda(Z-\E Z)}\right]\le \frac{(a\E Z+b)\lambda^2}{2(1-c_+ \lambda)}\]
For all $t> 0$,
\[\PP\{Z\ge \E Z + t\}\le \exp\left(-\frac{t^2}{2(a\E Z + b + c_+ t)}\right).\]
If $f$ is weakly $(a,b)$-self-bounding  and for all $i\le n$, all $x\in \Lambda$, $f_i(x^{(i)})\le f(x)$, then for all $0\le \lambda\le 2/a$,
\[\log\E\left[e^{\lambda(Z-\E Z)}\right]\le \frac{(a\E Z+b)\lambda^2}{2(1-a\lambda/2)}\]
and for all $t>0$,
\[\PP\{Z\ge \E Z + t\}\le \exp\left(-\frac{t^2}{2(a\E Z + b + at/2}\right).\]
If $f$ is weakly $(a,b)$-self-bounding  and $f(x)-f_i(x^{(i)})\le 1$ for each $i\le n$ and $x\in \Lambda$, then for $0\le t\le \E Z$,
\[\PP\{Z\le \E Z-t\}\le \exp\left(-\frac{t^2}{2(a\E Z + b + c_- t)}\right).\]
\end{thm}

We define $\alpha$-self-bounding functions as in \citep{Martoncoupling}:

\begin{definition}\label{defalphasb}
Let $\Omega=\Omega_1\times\ldots\times \Omega_{n}$. Let $a,b\ge 0$. 

\begin{enumerate}
\item
We say that $f:\Omega\to \R$ is \emph{$\alpha$-$(a,b)$ self-bounding} if there is $\alpha:\Lambda\to \R_+^{n}$ such that
\begin{enumerate}
\item $f(x)-f(y)\le \sum_{i\le n} \alpha_i(x)\II[x_i\ne y_i]$ for every $x,y\in \Omega$.\\ 
\item $\alpha_i(x)\le 1$ for every $i\le n, x\in \Omega$.\\
\item $\sum_{i\le n} \alpha_i(x)\le a f(x)+b$.
\end{enumerate}
\item
We say that $f:\Omega\to \R$ is  \emph{weakly $\alpha$-$(a,b)$ self-bounding} if there is $\alpha:\Lambda\to \R_+^{n}$ such that
\begin{enumerate}
\item $f(x)-f(y)\le \sum_{i\le n} \alpha_i(x)\II[x_i\ne y_i]$ for every $x,y\in \Omega$.\\ 
\item $\sum_{i\le n} \alpha_i(x)^2\le a f(x)+b$.
\end{enumerate}
\end{enumerate}
\end{definition}
\begin{remark}
It is easy to see that $\alpha$-$(a,b)$ self-bounding functions are also weakly $\alpha$-$(a,b)$ self-bounding.
\end{remark}

\begin{remark}
The following relations hold:

\begin{tabular}{c c c}
$(a,b)$-self-bounding & $\Rightarrow$ & weakly $(a,b)$-self-bounding\\
$\Uparrow$& &$\Uparrow$\\
$\alpha$-$(a,b)$-self-bounding  & $\Rightarrow$ & weakly $\alpha$-$(a,b)$-self-bounding
\end{tabular}\\\\
The reverse implications are false in general.
\end{remark}

\section{Results}
The following theorem bounds the moment generating function of subadditive functions of (LD1) variables.

\begin{thm}\label{thmsubadd}
Suppose that $\Lambda:=\R^n$ (or $\C^n$), and $f:\Lambda\to \R$ is a subadditive function, i.e. for any $x,y\in \Lambda$, $f(x+y)\le f(x)+f(y)$. 

Let $\A=[n]$, and suppose that $\{X_i\}_{i\in \A}$ satisfy (LD1). Then
\begin{itemize}
\item
If $\{\A_j\}_j$ is one of the smallest proper covers of $\A$, having $\chi(\A)$ elements, then for $\theta>0$,
\begin{equation}\label{chiAeq}
\E\left(e^{\theta f(X_1,\ldots,X_n)}\right)\le \frac{1}{\chi(\A)}\sum_{j=1}^{\chi(\A)}\E\left(e^{\theta \chi(\A) f\left(X_{\A_j}\right)}\right),
\end{equation}
here $X_{\A_j}\in \Lambda$ is defined by replacing all the components of $X$ with zeros outside of $\A_j$.
\item Suppose that $f$ also satisfies $f(cx)\le cf(x)$ for every $0\le c\le 1$, $x\in \Lambda$. Let $\{(A_j,w_j)\}_j$ be an exact fractional cover with $\sum_{j}w_j=\chi^*(\A)$.
Then for every $\theta>0$,
\begin{equation}\label{chistarAeq}
\E\left(e^{\theta f(X_1,\ldots,X_n)}\right)\le \frac{1}{\chi^*(\A)}\sum_j w_j \E\left(e^{\theta \chi^*(\A) f\left(X_{\A_j}\right)}\right).
\end{equation}
\end{itemize}
\end{thm}

Our next theorem is about (HD,$k$,$l$) dependent random variables.

\begin{thm}\label{thmHD}
Let $X=(X_1,\ldots,X_n)$ be a $\Lambda$ valued random vector, satisfying (HD,$k$,$l$) for some $Y=(Y_1,\ldots,Y_N)$.
\begin{itemize}
\item
If $f:\Lambda\to \R$ is $\alpha$-$(a,b)$-self-bounding for some $a,b\ge 0$,
then there is a function $g$ such that $f(X)=g(Y)$ almost surely, and the function $h$ defined as $h(y):=g(y)/l$ is $\left(ka,\frac{k}{l}b\right)$-self-bounding.
\item
If $f:\Lambda\to \R$ is weakly $\alpha$-$(a,b)$-self-bounding for some $a,b>0$, then there is a function $g$ such that $f(X)=g(Y)$ almost surely, and the function $h$ defined as $h(y):=g(y)/l$ is weakly $\left(k a,\frac{k}{l} b\right)$-self-bounding, and satisfies $|h(z)-h(z')|\le 1$ for any $z,z'$ differing only in one coordinate.
\end{itemize}
\end{thm}

Let $\Lambda:=\Lambda_1\times \ldots\times \Lambda_n$, we say that a function $f:\Lambda\to \R$ is $c$-weighted Hamming Lipschitz for some $c\in \R_+^n$ if for any $x,y\in \Lambda$ only differing in coordinate $i$, $|f(x)-f(y)|\le c_i$.

\begin{corollary}\label{McdiarmidHD}
Suppose that $X=\{X_i, i\in [n]\}$ satisfies (HD,$l$,$k$),  $X\in \Lambda$, then for any $c$-weighted Hamming Lipschitz $f:\Lambda\to \R$, we have for every $\lambda>0$,
\begin{eqnarray}
\E\left(\exp\left(\lambda [f(X)-\E f(X)]\right)\right)\le \exp\left(\frac{\lambda^2 kl \sum_{i=1}^n c_i^2}{8}\right),
\end{eqnarray}
and thus for every $t\ge 0$
\begin{equation}
\PP(f(X)-\E f(X)\ge t), \PP(f(X)-\E f(X)\le -t) \le \exp\left(\frac{-2 t^2}{kl \sum_{i=1}^n c_i^2}\right).
\end{equation}
\end{corollary}

\begin{corollary}\label{TalagrandHD}
Suppose that $X=\{X_i, i\in [n]\}$ satisfies (HD,$l$,$k$), then a version of Talagrand's convex distance inequality holds:
\begin{equation}\label{talexp14HD}
\E \left(\exp\left(\frac{1}{10kl}d_T^2(X,S)\right)\right)\le \frac{1}{\PP(X\in S)}.
\end{equation}
 and as a consequence,
\begin{equation}\label{talSStHD}
\PP(X\in S)\PP\left(X\in \ol{S_t}\right)\le \exp\left(-t^2/(10kl)\right).
\end{equation}
\end{corollary}

\section{Applications}

Random matrix models with dependent entries have been considered by several authors, see for example \citep{Zeitounifiniterange}.  A model where the entries are (LD1) type, and satisfy some additional condition, appears in \citep{Schenker}, and then was further developed in \citep{Stolz}. These results  show asymptotic convergence of the eigenvalue distribution to circular law or the singular value distribution to Marchenko-Pastur law. In this paper, we prove some non-asymptotic results. In our first example, we show concentration for the upper tail of the norm of a random matrix with (LD1) entries:

\subsection{Norm of a random matrix with (LD1) entries}
\begin{thm}\label{normld1}
Let $M:=(X_{i,j})_{1\le i,j\le n}$ be a Hermitian matrix with entries bounded by $K$ in absolute value, and the upper diagonal entries
$(X_{i,j})_{1\le i\le j\le n}$ satisfying (LD1) with neighborhoods $\A=\{\A_{i,j}\}_{1\le i\le j\le n}$.
Then 
\begin{equation}\label{normconc}
\PP(||M||\ge 3C\chi^*(\A) K \sqrt{n}+t)\le \exp\left(-\frac{t^2}{32 \chi^*(\A)^2 K^2}\right)
\end{equation}
here $C$ is the universal constant in \citep{latala2005some}.

If $M:=(X_{i,j})_{1\le i\le n,j\le N}$ is a complex valued matrix, with entries satisfying  (LD1) with neighborhoods $\A=\{\A_{i,j}\}_{1\le i\le j\le n}$, then
\begin{equation}\label{normconc2}
\PP\left(||M||\ge C\chi^*(\A) K \left(\sqrt{n}+\sqrt{N}+\sqrt[4]{nN}\right)+t\right)\le \exp\left(-\frac{t^2}{8\chi^*(\A)^2 K^2}\right)
\end{equation}
\end{thm}

\subsection{Eigenvalues of a random matrix with (HD,$k$,$l$) entries}
The following theorem is a generalization of Theorem 1 of \citep{Aloneig} to this setting.

\begin{thm}\label{HDeigconc}
Let $M$ be a real valued random symmetric matrix with entries bounded by 1, and the upper diagonal entries satisfying (HD,$k$,$l$). Let $\lambda_1(M)\ge \ldots\ge \lambda_n(M)$ be the eigenvalues of $M$ in decreasing order.
For every positive integer $1\le s\le n$, the probability that $\lambda_s(M)$ deviates from its median by more than $t$ is at most $4e^{-t^2/(80 s^2\cdot kl)}$. The same estimate holds for the probability that $\lambda_{n-s+1}$ deviates from its median by more than $t$.
\end{thm}
\begin{remark}
We leave it to the reader as an exercise to adapt the proof of \citep{meckes2004concentration}, Theorem 2 to this setting, and reduce the $s^2$ in the exponent to $s$.
\end{remark}
\begin{remark}
The correct range of concentration of the eigenvalues of a Wigner matrix $M_n$ is $O\left(\sqrt{\frac{\log(n)}{n}}\right)$ in the bulk, and $O\left(n^{-1/6}\right)$ on the edge, as it is shown in \citep{dallaporta2012eigenvalue}.
\end{remark}

\section{Proofs}\label{secproofs}
\begin{proof}[Proof of Lemma \ref{HDlemma}]
Given $Y_1,\ldots,Y_N$ and $S_1,\ldots,S_n$, let us define $S_1',\ldots,S_n'$ the following way: for evey $1\le i\le N$, $i\in S_j'$ if and only if $1\le j\le n$ is the smallest index such that $i\in S_j$. Now let $Y_1':=Y_{S_1'},\ldots,Y_n':=Y_{S_n'}$, then the reader can easily verify that $X_1,\ldots,X_n$ satisfies (HD,$m$,$m$) for $\{Y_i', i\in [n]\}$.
\end{proof}

\begin{proof}[Proof of Theorem \ref{thmsubadd}]
Part 1 is implied by the subadditivity of $f$, and the convexity of the exponential function. For $\theta>0$,
\begin{eqnarray*}
&&e^{\theta f(X)}\le e^{\theta \left(f\left(X_{\A_1}\right)+\ldots+f\left(X_{\A_{\chi(\A)}}\right)\right)}=e^{\theta \chi(\A) \sum_{i=1}^{\chi(\A)}\frac{1}{\chi(\A)}f\left(X_{\A_i}\right)}\\
&&\le \frac{1}{\chi(\A)}  \sum_{i=1}^{\chi(\A)} e^{\theta \chi(\A)f\left(X_{\A_i}\right)}
\end{eqnarray*}
Taking expectations gives the result. Part 2 is similar: for $\theta>0$,
\begin{eqnarray*}
&&e^{\theta f(X)}=e^{\theta f\left(\sum_{i} w_i X_{\A_i}\right)}\le e^{\theta \sum_i f\left(w_i X_{\A_i}\right)} \le e^{\theta \sum_i w_i f\left(X_{\A_i}\right)}\\
&&\le e^{\theta \chi^*(\A) \sum_i \frac{w_i}{\chi^*(\A)} f\left(X_{\A_i}\right)}\le \sum_{i}\frac{w_i}{\chi^{*}(\A)} e^{\theta \chi^*(\A) f\left(X_{\A_i}\right)}.
\end{eqnarray*}
Taking expectations gives the result. We have used the $f(cx)\le cf(x)$ condition (for $0\le c\le 1$, since $0\le w_i\le 1$ for exact covers).
\end{proof}

\begin{proof}[Proof of Theorem \ref{thmHD}]
We are only going to prove the first part (concerning $\alpha$-self-bounding functions), the second part is similar.

The existence of $g$ such that $f(X)=g(Y)$ is trivial, since $X$ is a function of $Y$. For each $1\le i\le N$, $y\in \Sigma$, let 

\[h_i(y_{-i}):=\inf_{y_i'}h(y_1,\ldots,y_{i-1},y_{i}',y_{i+1},\ldots,y_n).\]
Using the fact that $f(X)$ is $\alpha$-$(a,b)$ self-bounding, we can write
\begin{eqnarray*}
&&\sum_{i=1}^N h(y)-h_i(y_{-i})\le \frac{1}{l}\sum_{i=1}^N\sum_{j\in R_i}\alpha_j(x(y))\\
&&\le\frac{1}{l} k \sum_{j=1}^n \alpha_j(x(y))\le \frac{k}{l} (af(x(y))+b)\le k h(y)+\frac{k}{l}b,
\end{eqnarray*}
and thus the result follows.
\end{proof}

\begin{proof}[Proof of Corollary \ref{McdiarmidHD}]
This follows by a 4 times worse constant from the fact that $f$ is weakly $\alpha$-$(0,\sum_{i=1}^n c_i^2)$  self-bounding. We can get this better constant by writing $f(X)=g(Y)$, and directly applying Mcdiarmid's bounded differences inequality to the independent variables $Y_1,\ldots,Y_N$.
\end{proof}

\begin{proof}[Proof of Corollary \ref{TalagrandHD}]
The proof is similar to the proof of Corollary 1 of \citep{LugosiSelfBounding}. By Lemma 3.2 of \citep{SelfBoundingDobrushin}, we know that $d_T^2(x,S)$ is weakly $\alpha$-(4,0) self-bounding. By Theorem \ref{thmHD} we have a function $g$ with $g(Y)=d_T^2(X,S)$, and $h(y)=g(y)/l$ is $(4k,0)$ self-bounding. 

Thus, by Theorem \ref{Lugosi}, we have that for $0\le \lambda\le \frac{1}{2kl}$,
\begin{eqnarray*}
&&\log \E(e^{\lambda \left(d_T^2(X,S)-\E(d_T^2(X,S))\right)})=\log \E(e^{\lambda l \left(h(Y)-\E(h(Y))\right)})\\
&&\le \frac{4kl \lambda^2 \E(d_T^2(X,S))}{2(1-2kl\lambda)},
\end{eqnarray*}
thus
\begin{equation}\label{dTeq1}
\log \E(e^{\lambda d_T^2(X,S)})\le \lambda \E(d_T^2(X,S)) + \frac{4kl \lambda^2}{2(1-2kl\lambda)} \E(d_T^2(X,S)).
\end{equation}
Again by Theorem \ref{Lugosi},
\begin{eqnarray*}
&&\log \PP(X\in S)= \log \PP(d_T^2(X,S) -\E d_T^2(X,S) \le - \E d_T^2(X,S))\\
&&=\log \PP(h(Y) -\E h(Y) \le - \E h(Y))\le -\frac{\E h(Y)}{2\cdot 4k}=
-\frac{\E d_T^2(X,S)}{8kl}.
\end{eqnarray*}
By adding this to \eqref{dTeq1} for $\lambda=\frac{1}{10kl}$, we get the result.
\end{proof}

Before proving Theorem \ref{normld1}, we introduce some results that we are going to use:
\begin{thm}[Theorem 2 of \citep{latala2005some}]\label{latalathm}
For any finite matrix of independent mean zero r.v.'s $X_{ij}$ we have
\[\E||(X_{ij})||\le C\left(\max_{i}\sqrt{\sum_j \E X_{ij}^2}+\max_j \sqrt{\sum_i \E X_{ij}^2}+\sqrt[4]{\sum_{ij}\E X_{ij}^4}\right)\]
\end{thm}

\begin{proposition}\label{eigvalprop}
Let $A_{i,j}$ be a symmetric real valued matrix with entries bounded by $K$. Let $\lambda_1(A_{i,j})$ be its largest eigenvalue. If we look at $\lambda_1$ as a function of only $A_{1\le i\le j\le n}$, then $\lambda_1$ is weakly $(0,16K^2)$ self-bounding.
\end{proposition}
\begin{proof}
This is a reformulation of Example of page 42-43 of \citep{Lugosinotes} (but they made a mistake by treating all the elements of the symmetric matrix as independent random variables, so the correct constant is 4 times worse).

For some $v\in \R^n$,
\[\lambda_1(A)=\sup_{u\in \R^n:||u||=1}u^t A u= v^t Av=\sum_{i,j} v_i v_j A_{ij},\]
and for $A'$ created by replacing $A_{i,j}$ and $A_{j,i}$ with $A_{i,j}'$, we have
\[\lambda_1(A)-\lambda_1(A')\le 2 v_i v_j (A_{i,j}-A_{i,j}')\le 4K |v_i| |v_j|,\]
so the result follows by 
\[\sum_{i,j}(4K |v_i| |v_j|)^2\le 16K^2\left(\sum_i |v_i|^2\right) \left(\sum_j |v_j|^2\right)\le 16K^2.\]
\end{proof}

In the following proposition, we prove a similar result for largest singular value:
\begin{proposition}\label{normhermprop}
Let $A:=(A_{i,j})$ be an $n\times N$ sized complex valued matrix with entries bounded by $K$ in absolute value. Let $s_1(A)$ be its largest singular value (which is equal to its operator norm). Then $s_1$ is weakly $(0,4K^2)$ self-bounding.

If $A:=(A_{i,j})$ is a Hermitian $n\times n$ matrix, then $s_1$, as the function of only $A_{1\le i\le j\le n}$, is weakly $(0,16K^2)$ self-bounding.
\end{proposition}
\begin{proof}
We can write, for some $U,V$ complex valued unit vectors, that
\begin{eqnarray*}
s_1(A)&=&\sup_{u,v\in \C^n: ||u||,||v||=1} Re[ u^* A v]=\sup_{u,v: ||u||,||v||=1}\sum_{ij} Re [u_i^* v_j A_{ij}]\\
&=&\sum_{ij}Re [U_i^* V_j A_{ij}].
\end{eqnarray*}
Let's denote by $A'$ the matrix that we get by changing $A_{ij}$ to $A_{ij}'$, then
one can easily see that
\[s_1(A)-s_1(A')\le Re [U_i^* V_j A_{ij}]-Re [U_i^* V_j A_{ij}']\le 2 K |U_i| |V_j|,\]
and thus the first statement of the proposition is implied by
\[\sum_{ij}(2 K |U_i| |V_j|)^2\le 4K^2\left(\sum_{1\le i\le n}|U_i|^2\right)\left(\sum_{1\le j\le N}|V_j|^2\right)\le 4K^2.\]
The proof of the last statement is similar, and is left to the reader.
\end{proof}

\begin{proof}[Proof of Theorem \ref{normld1}]
In the Hermitian case, let $\Lambda= \C^{n(n-1)/2}$. Let $f:\Lambda\to \R$ be the norm of the Hermitian matrix with upper diagonal elements as argument. Let $\A=[n(n-1)/2]$, and let $(\A_i,w_i)$ be an exact cover for $\A$ with $\sum_i w_i=\chi^*(A)$.

Then $f(cx)= cf(x)$ for $0\le c\le 1$, so by the second part of Theorem \ref{thmsubadd}, 

\[\E\left(e^{\theta ||M||}\right)\le \frac{1}{\chi^*(\A)}\sum_j w_j \E\left(e^{\theta \chi^*(\A) \left|\left|M_{\A_j}\right|\right|}\right).\]

Now for each $j$, $M_{\A_j}$ is a symmetric matrix with independent entries, and thus,
by the second statement of Proposition \ref{normhermprop}, we know that $||M_{\A_j}||$ is weakly $(0,16K^2)$ self-bounding as a function of its upper diagonal entries. This means that by Theorem \ref{Lugosi}, we have for every $\lambda>0$, for every $j$,
\[\E\left(e^{\lambda \left|\left|M_{\A_j}\right|\right|}\right)\le e^{\lambda \E \left|\left|M_{\A_j}\right|\right|}\cdot e^{8K^2\lambda^2},\]
which, by setting $\lambda:=\chi^*(\A)\theta$, implies that
\begin{equation}\label{Mchisqeq}
\E\left(e^{\theta ||M||}\right)\le  e^{8K^2\chi^*(\A)^2\theta^2} \frac{1}{\chi^*(\A)}\sum_j w_j e^{\theta \chi^*(\A) \E \left|\left|M_{\A_j}\right|\right|}.
\end{equation}
Now, Theorem \ref{latalathm}, combined with the boundedness assumption, implies that
\[\E \left|\left|M_{\A_j}\right|\right|\le 3C\sqrt{n},\]
thus by Markov's inequality, we get
\[\PP\left(||M||\ge 3C\chi^*(\A) K \sqrt{n}+t\right)\le \exp\left(8K^2\chi^*(\A)^2\theta^2-\theta t\right),\]
and taking $\theta=\frac{t}{16\chi^*(\A)^2}$ gives the bound.

The proof of the rectangular case is similar. 
\end{proof}

\begin{proof}[Proof of Theorem \ref{HDeigconc}]
This is just a simple adaptation of the argument of \citep{Aloneig}. The version of Talagrand's inequality for (HD,$k$,$l$) dependence, which follows from Corollary \ref{TalagrandHD}, is of the form
\[Pr[\A] Pr[\B]\le e^{-t^2/(10kl)},\]
and thus the constant the exponent becomes $80kl$ (instead of $32$).
\end{proof}

\section*{Acknowledgements}
The author thanks Doma Sz\'{a}sz and Mogyi T\'{o}th for infecting him with their enthusiasm of probability.
He thanks his thesis supervisors, Louis Chen and Adrian R\"{o}llin, for the opportunity to study in Singapore, and their  useful advices. Finally, many thanks to my brother, Roland Paulin, for the enlightening discussions.
\bibliographystyle{imsart-nameyear}
\bibliography{References}

\end{document}